\documentclass[12pt]{amsart}

\usepackage{mathrsfs}
\usepackage{amssymb}
\usepackage{amsthm}
\usepackage{dsfont}
\usepackage{amssymb}
\usepackage{amsmath,amscd}
\usepackage[all,cmtip]{xy}
\usepackage{mathrsfs}
\usepackage{multirow}
\usepackage{hyperref}
\usepackage{graphicx}
\usepackage{amsfonts, amsthm}
\usepackage{fullpage}
\usepackage{enumerate}
\usepackage{appendix}
\usepackage{cite}

\theoremstyle{plain}
\newtheorem{thm}[subsection]{Theorem}
\newtheorem{lemma}[subsection]{Lemma}
\newtheorem{prop}[subsection]{Proposition}
\newtheorem{cor}[subsection]{Corollary}

\theoremstyle{definition}
\newtheorem{rmk}[subsection]{Remark}
\newtheorem{defn}[subsection]{Definition}

\let\a\alpha

\let\c\chi

\let\g\gamma

\let\l\lambda
\let\m\mu

\let\o\omega
\let\r\rho
\let\s\sigma
\let\t\tau

\def\scr{\mathscr}
\def\cal{\mathcal}
\let\D\Delta

\let\G\Gamma

\let\O\Omega

\newcommand{\ra}{\longrightarrow}

\newcommand{\im}{{\rm im}\:}

\newcommand{\DD}{{\mathbb D}}
\newcommand{\GG}{{\mathbb G}}
\newcommand{\F}{{\mathbb F}}
\newcommand{\Z}{{\mathbb Z}}
\newcommand{\C}{{\mathbb C}}
\newcommand{\Q}{{\mathbb Q}}
\newcommand{\R}{{\mathbb R}}

\newcommand{\PP}{{\mathbb P}}

\newcommand{\cC}{\mathcal{C}}

\newcommand{\cF}{\mathcal{F}}
\newcommand{\cG}{\mathcal{G}}
\newcommand{\cE}{\mathcal{E}}
\newcommand{\cI}{\mathcal{I}}
\newcommand{\cL}{{\mathcal{L}}}
\newcommand{\cO}{\mathcal{O}}
\newcommand{\cT}{\mathcal{T}}
\newcommand{\cV}{\mathcal{V}}
\newcommand{\cW}{\mathcal{W}}

\newcommand{\Spec}{{\mbox{Spec }}}

\newcommand{\End}{{\mathrm{End}}}

\newcommand{\cri}{{\mathrm{cris}}}

\newcommand{\Aut}{{\mathrm{Aut}}}

\newcommand{\ord}{{\mathrm{ord}}}
\newcommand{\Rep}{{\mathrm{Rep}}}

\newcommand{\univ}{{\mathrm{univ}}}
\newcommand{\Hom}{{\mathrm{Hom}}}

\normalfont\upshape
\begin{document}

\title{Crystalline Hodge cycles and Shimura curves in positive characteristics}
\author{Jie Xia}
\maketitle

\begin{abstract}
In this paper, we seek an appropriate definition for a Shimura curve of Hodge type in positive characteristics, i.e. a characterization of curves in positive characteristics which are reduction of Shimura curve over $\C$. Specifically, we study the liftablity of a curve in moduli space $\cal{A}_{4,1,k}$ of principally polarized abelian varieties over $k, \text{char } k=p$. We show that some conditions on the crystalline Hodge cycles over such a curve imply that this curve can be lifted to a Shimura curve. 
\end{abstract}

\section{Introduction}

\subsection{Motivations and the main results}
Classical Shimura varieties are defined over number fields. As far as I know, there is no definition of Shimura varieties in positive characteristics. The theory of integral models of Shimura varieties has been developed and it is the main tool to study Shimura varieties in positive characteristic.  In this paper, we seek an intrinsic definition of Shimura varieties in characteristic $p$ which guarantees that they are reductions of classical Shimura varieties. 

Let us recall the definition of Shimura varieties. 
Let $\mathbb{S}$ be the Weil restriction of scalar $\mathrm{Res}_{\C/\R} \GG_m$. A Shimura datum $(G, X)$ consists of a reductive group $G$ defined over $\Q$ and a $G(\R)$-conjugate class $X$ of a cocharacter $h: \mathbb{S} \ra G_\R$ such that 
\begin{enumerate}
\item For any $h$ in $X$, the Lie algebra $\frak{g}$ of $G_\R$, viewed as conjugation representation of $\mathbb{S}$ via $h$, has the type $(1,-1), (0,0)$ and $(-1,1)$.
\item The adjoint action of $h(i)$ induces a Cartan involution on the adjoint group of $G_\R$.
\item The adjoint group of $G_\R$ does not have a factor $H$ defined over $\Q$ such that the projection of $h$ on $H_\R$ is trivial.
\end{enumerate}

For any sufficiently small compact open subgroup $K$ of $G(\mathbb{A}_f)$, 
\[Sh_K(G, X)=G(\Q)\backslash X \times G(\mathbb{A}_f) /K\] is a complex algebraic variety. Take the inverse limit over $K$ and the resulting inverse system is called a \textit{Shimura variety}.

A large class of Shimura varieties admits an interpretation in terms of moduli of certain polarized abelian varieties. For instance, \textit{Shimura varieties of PEL type} parametrize polarized abelian varieties with a prescribed endomorphism ring. In such a class, we can take the moduli description as an equivalent definition of Shimura varieties in positive characteristics. 

In \cite{famofav}, Mumford defines a class of Shimura varieties, containing PEL type, as the moduli scheme of polarized abelian varieties (up to isogeny) whose Hodge groups are contained in a prescribed Mumford-Tate group. We call this class \textit{Shimura varieties of Hodge type}. In this paper, we focus on an important example of Shimura varieties of Hodge type. 

In the paper \cite{Mum}, Mumford constructs smooth proper Shimura curves. With appropriate level structures, such curves parameterize complex abelian fourfolds with certain Hodge cycles. In particular, the generic points on any such curve correspond to abelian varieties with endomorphism ring isomorphic to $\Z$. Therefore the Shimura curves are not of PEL type. We call such Shimura curves (with the family of abelian fourfolds) the Mumford curves. 

%An abelian fourfold is of Mumford type if it appears to be a fiber in one of the Mumford curves.

Moonen and Zarhin(\cite{Moon}) study the Hodge cycles on complex simple abelian fourfolds and give a thorough classification. 
\begin{defn}
A Hodge class is exceptional if it is not generated by divisor classes.
\end{defn}

From the Table I of the paper \cite{Moon}, we know that for a simple complex abelian fourfold $X$,  the following are equivalent: 
\begin{enumerate}
\item $X$ is of Mumford type ,
\item $X$ has no exceptional Hodge class yet has exceptional Hodge classes in $H^4(X^2)$.
\end{enumerate}
In particular, the second condition implies $\dim_\C H^{2,2}(X)=1$. 
This elegant result motivates us to study whether this criterion is still true globally in positive characteristics. Instead of a single abelian variety, we consider a family of polarized abelian varieties over a proper curve. We wonder that whether the existence of some special crystalline Hodge cycles over the curve can imply that it is a reduction of a Shimura curve. 

In the paper \cite[4.1.2]{Hodge cycles}, Ogus defines the crytalline Hodge cycles of a proper smooth scheme $Z$ over the Witt ring $W$ to be the Frobenius invariant of $H_\cri(Z/W)$(with weight 0). In our case, for any abelian scheme over a proper curve of positive characteristic,  crystalline Hodge cycles are defined to be the Frobenius eigen-elements of the global sections of Dieudonne crystal associated to the abelian scheme. 

Our basic notations and main theorem are as follows.

Let $k$ be an algebraic closure of $\F_p$, $C$ be a proper smooth curve over $k$ with absolute Frobenius $\sigma$ and $\pi: X \ra C$ be a family of four dimensional principally polarized abelian varieties. Since $\pi $ is of finite type, it can be defined over a finite field $\F_{p^{f_0}} \subset k$. 

Let $\cE$ be the Dieudonne crystal $R^1\pi_{\cri,*}(\cO_X)$ with Frobenius $F$ and Verschiebung $V$. These notations are fixed till the end of the paper. We have defined weak Mumford curves in \cite{Xia3}.

\begin{thm} \label{main thm}
Notation as above. Assume $p>2$ and there exists an integer $f$ such that $f$ is a multiple of $f_0$ and 
\begin{enumerate}
\item $X_c$ is ordinary for some closed point $c\in C$,
\item $\dim_{\Q_{p^f}} \G((C/W(k))_\cri, \wedge^4 \cE) ^{F^{f}-p^{2f}} \otimes \Q_{p^f}=1$,
\item  $\G((C/W(k))_\cri, \scr{E}nd(\wedge^2 \cE))^{F^f} \otimes \Q_{p^f}\cong \Q^{\times 4}_{p^f}$ as algebras. 
\end{enumerate}
Then $X \ra C$ is a weak Mumford curve over $k$.

If we further assume the Higgs field of $\cE$ is maximal, then there exists a family of abelian fourfolds $Y\ra C'$ such that 
\begin{enumerate} [(a)]
\item $C' \ra C$ is a finite \'etale covering,
\item $Y \ra X'$ is an isogeny over $C'$, between $Y$ and the pullback family of $X$,
\item $Y \ra C'$ is a good reduction of a Mumford curve.
\end{enumerate}
\end{thm}

\subsection{Overview of the proof}

In view of the main theorem in our previous paper \cite{Xia3}, it is enough to prove the following result. 
\begin{prop} \label{the first step}
Notations as \ref{main thm}. Under the conditions (2) and (3) of Theorem \ref{main thm}, there exists a finite \'etale covering $C'$ of $C$, such that the pull-back $\cE'$ of $\cE$ has a tensor decomposition: 
\[\cE' \cong \cV_1 \otimes \cV_2 \otimes \cV_3\]
 as $F^f$-isocrystals where $\cV_i$ are rank 2 isocrystals over $C'$.
\end{prop}

To prove \ref{the first step}, note under the Tannakian formalism, the $F^f$-crystal $\cE$ gives rise to a linear algebraic group $G_E$ acting on an 8-dimensional vector space $E$. The group $G_E$ is reductive, which follows from a general result (see \ref{simple and irreducible}).

Conditions (2) and (3) translate to 
\[ \dim_{B(k_0)}(\wedge^4 E)^{G_E}=1, \End(\wedge^2 E)^{G_E} \cong \Q^{\times 4}_{p^f}. \]

Using the classification of representations of simple Lie algebras, we can show $G_E$ geometrically has the shape of $SL(2)^{\times 3}$ and $E$ geometrically corresponds to a tensor product of three copies of standard representations of $SL(2)$.

The obstruction to descend the result to the base field is a certain cohomology class (see \ref{bundle},\ref{connection} and \ref{F^f-isocrystal}). It is where we have to make a finite \'etale cover of $C$. 

\subsection{Structure of the paper}

In Section \ref{example}, we demonstrate the theory of this paper is non-empty by showing examples satisfy all the conditions in \ref{main thm}.

Section \ref{prerequisite and notations} contains some preliminary results and basic notations.

In Section \ref{a glance}, we mainly prove \ref{simple and irreducible}, that over a proper smooth curve, a simple family of abelian varieties corresponds to a reductive group and an irreducible representation in the Tannakian formalism. 

Section \ref{the choice} and \ref{the structure} are mainly devoted to show how to choose the finite \'etale covering to kill the obstruction and obtain the tensor decomposition.
We finish the proof of \ref{main thm} in Section \ref{the end}.

In Section \ref{a variation},  we prove a variation of \ref{main thm} in which we drop the action of Frobenius.

\subsection*{Acknowledgements:} I thank my advisor Johan de Jong for suggesting this project. Without his generous help and guidance, this paper would not exist.

\section{The crystals associated to a Barsotti-Tate group}
 We explain the concepts involved in Theorem \ref{main thm} and state some results on crystals and Barsotti-Tate groups which we will use later. 
\subsection{}
The curve $C/k$ in \ref{main thm} has a natural crystalline site $\cri(C/W(k))$. The higher direct image $\cE=R^1\pi_{\cri,*}(\cO_X)$ of the abelian scheme $\pi: X\ra C$ is a crystal in locally free sheaves. 

\begin{defn}
A \textit{Dieudonne crystal} $\cF$ in $\cri(C/W(k))$ satisfies 
\begin{enumerate}
\item it is a crystal in locally free sheaves,
\item there exist $F: \cF^\sigma \ra \cF$ and $V: \cF \ra \cF^\sigma$ such that $F\circ V=p. \mathrm{Id}$, $V \circ F=p. \mathrm{Id}.$
\end{enumerate}
\end{defn}

The crystal $\cE$ is further a Dieudonne crystal. 

Choose an arbitrary lifting $\tilde C$ of $C$ to $W(k)$. The category of crystals in locally free sheaves in $\cri(C/W(k))$ is equivalent to the category of vector bundles with an integrable connection on $\tilde C$. In particular, choosing an open affine subset $U\subset C$ and a lifting $\tilde U$ of $U$, we have a lifting of Frobenius $\tilde \s$ over $\tilde U$. 

In the rest of the paper, by crystal, we mean a crystal in locally free sheaves. Therefore an $F$-isocrystal on $\cri(U/W(k))$ corresponds to a triple $(M, \nabla, F)$, a sheaf of module on $\tilde U$ with an integrable connection and Frobenius $F: M^{\tilde \s} \ra M$.

\subsection{} 
A Barsotti-Tate (BT) group $G$ over $C$ is a $p$-divisible, $p$-torsion and the $p$-kernel is a finite locally free group scheme.  Each $p^i$-kernel $G[p^i]$ is a truncated BT group. By \cite{BBM}, the crystalline Dieudonne functor $\DD(G)=\scr{E}xt^1(\underline{G}, \cO_C)$ associates a Dieudonne crystal over $\cri(C/W(k))$ to a BT group.  And $\DD(G[p])=\DD(G)_C$ admits a filtration 
\[0\ra \o_G \ra \DD(G)_C\ra \a_G \ra 0.\] 

In the context of the Theorem \ref{main thm}, $\cE=\DD(X[p^\infty])$ and the filtration on $\cE_C$ is just the Hodge filtration of $X/C$: $\o=\pi_* \O_{X/C}, \a=R^1\pi_*(\cO_X)$. In particular, $\cE_C$ has the Gauss-Manin connection and it induces a $\cO_C$-linear map, called Higgs field: 
\[\theta: \o \ra \a\otimes \O^1_C\] which is related to Kodaira-Spencer map. The Higgs field can be defined alternatively: 
\[\xymatrix{
\o \ar[r] \ar@{-->}[drr] & \cE_C \ar[r] \ar[d]^\nabla & \a \\
&\cE_C \otimes \O^1_C \ar[r] & \a\otimes \O^1_C.
}
\] The other is from the long exact sequence of 
\[0\ra \pi^*\O_C \ra \O_X \ra \O_{X/C} \ra 0\] and the boundary map $\cdots \ra \pi_*\O_{X/C} \xrightarrow{\partial} R^1\pi_* \cO_X \otimes \O^1_C \ra \cdots$ gives the Higgs field.
Condition (3) in \ref{main thm} just means the map $\theta$ is isomorphic.

\begin{thm}(\cite[Main Theorem 1]{deJ} ) \label{equiv on curve} The category of Dieudonne crystals over $\cri(C/W(k))$ is anti-equivalent to the category of BT groups over $C$. 
\end{thm}

\section{Example} \label{example}

To justify that we are not proving a vacuous theorem, in this section, we show some reductions of a Mumford curve satisfy all the conditions in Theorem \ref{main thm}.

First we quote a theorem from \cite{JX}: let $k$ be an algebraically closed field of characteristic $p$.
\begin{thm} (\cite[ Theorem 1.2]{JX}) \label{my thm}

For infinitely many prime $p$, there exists a family of principally polarized abelian fourfolds over a smooth proper curve $\tilde f: \tilde X \ra \tilde C$ over $W(k)$ such that 
\begin{enumerate}
\item $(\tilde X \xrightarrow{\tilde f} \tilde C)\otimes \C $ is a Mumford curve,
\item the reduction $X \ra C$ of $\tilde X \ra \tilde C$ at $k$ is generically ordinary,
\item the Dieudonne crystal $\cE\cong \cV \otimes \cT$ where $\cV$ is a Dieudonne crystal of rank 2 with maximal Higgs field and $\cT$ is a unit root crystal of rank 4,
\end{enumerate}
\end{thm}

We use $\tilde X \xrightarrow{\tilde f} \tilde C$ to denote the Mumford curve defined over the Witt ring $W(k)$(see \cite[Section 2.5]{JX}).  Then $\cE$ corresponds to $R^1\tilde f_*(\O^._{\tilde X/\tilde C}) $. 
\begin{prop} (\cite{JX})
The unit root crystal $(\cT, F_T)$ has the tensor decomposition as crystals: $\cT \cong \cV_2 \otimes \cV_3$ and either of the following two cases is true \begin{enumerate}
\item  $\cV_2, \cV_3$ are $F$-crystals and the isomorphism respects the Frobenius.
\item  $F_T=F_2 \otimes F_3$ where $F_2: \cV^\s_2 \ra \cV_3$ and $F_3: \cV^\s_3 \ra \cV_2$.
\end{enumerate} 
\end{prop}
Therefore $\cT \cong \cV_2 \otimes \cV_3$ as $F^2$-isocrystals.

For Condition (3) in \ref{main thm}, the self product of the polarization gives 
\[\dim_{B(k_0)}\G(\tilde C, R^4\tilde f_*(\O^._{\tilde X/\tilde C})(2f))^{F^f} \otimes B(k_0) \geq 1\] for any integer $f$.

Base change to algebraically closed field $\bar B(k_0)$ and $\cE$ becomes a local system, with algebraic monodromy $G=SL(2)^{\times 3}$ (\cite{JX}, proposition 2.4). Hence $\G(\tilde C, \wedge^4 R^1 \tilde f_*(\O^._{\tilde X/\tilde C}))\otimes \bar B(k_0)=1$. We have 
\[\dim_{B(k_0)} \G(\tilde C, R^4\tilde f_*(\O^._{\tilde X/\tilde C}))^{F^f-q^{2f}} \otimes B(k_0) \leq 1.\]
Therefore (3) in \ref{main thm} is satisfied.

For the maximal Higgs field, since $\cE \cong \cV \otimes \cV_2 \otimes \cV_3$ as $F^{2f}$-isocrystals for any natural number $f$ and $\cV_1, \cV_2, \cV_3$ are irreducible as isocrystals, 
\[\wedge^2 \cE=S^2 \cV_2 \otimes S^2 \cV_3 \oplus S^2 \cV \otimes S^2 \cV_2 \oplus S^2 \cV \otimes S^2 \cV_3 \oplus \cO_C\] 
as direct sum of simple isocrystals. 
 Thereby the algebra $\End(\wedge^2 \cE)^{F^f}$ is isomorphic to $\Q^{\times 4}_{p^f}$ for some $f$. 

For Condition (1) in \ref{main thm}, it follows from (2) in Theorem \ref{my thm}.
%let $X/C$ be the special fiber of $\tilde X/\tilde C$ and then $\DD(X/C)=\cE$. Through the crystalline Dieudonne theory, it suffices to show as a F-isocrystal, $\cE$ is irreducible. Again by (\cite{JX}, proposition 2.4), the Tannakian duality of $\cE$ over $\C$ is a geometrically irreducible representation of $SL(2)^{\times 3}$. Therefore $X/C$ is simple. 

For Condition (2), since the universal family over the Shimura curve $ \tilde C$ has maximal Higgs field (see \cite[Theorem 0.9]{Moller}), the Higgs field of the special fiber $X/C$ is also maximal.

Therefore the special reduction of Mumford curve at $k$ satisfies all the conditions of Theorem \ref{main thm}.

\section{Prerequisite and notations} \label{prerequisite and notations}

\subsection{} We recall some general facts on Tannakian categories and then we focus on the Tannakian category of $F^f$-isocrystals, denoted as $F^f$-isoc$(C)$.

\begin{defn} Let $L$ be a field of characteristic 0.
A Tannakian category $T$ (over $L$) is a $L$-linear neutral rigid tensor abelian category with an exact fiber functor $\o: T\ra \rm{Vect}_L $. 
\end{defn}

\begin{thm} (\cite[Theorem 2.11]{Deli}) \label{Deli}
For any Tannakian category $T$, there exists an $L$-algebraic group $G$ such that $T$ is equivalent to $\Rep_L(G)$ as tensor categories. 
\end{thm}

\begin{rmk} \label{the equivalence}
Recall that $\pi$ is defined over finite field $\F_{p^{f_0}}$ and then for any $f$, a multiple of $f_0$, $\s^f$ fixes $\F_{p^{f_0}}$. Thereby in the category of $F^f$-isocrystals over $C$,  $\End(1)=\F_{p^{f_0}}$.  So the $F^f$-isocrystals over $C$ form a neutral Tannakian category over $\Q_{p^f}$ which is equivalent to $\Rep_{\Q_{p^f}}(G_\univ)$ for some group scheme $G_\univ$. 
\end{rmk}

For any finite \'etale covering $C'$ of $C$,  $F^f$-isocrystals over $C'$ also form a neutral Tannakian category, equivalent to $\Rep_{\Q_{p^f}}(G'_\univ)$. Let $\cC$(resp. $\cC'$) denote the category of $F^f$-isocrystals over $C$(resp. $C'$). Then the pullback induces $\cC \ra \cC'$. 

Let 
\[\G=\Aut(C'/C)\]
 be the finite automorphism group.  Then the action of $\G$ on $C'$ induces 
\[\G \ra \Aut^\otimes(\cC'), \g \mapsto \g^*.\] Obviously $1\in \G$ induces the identity on $\cC'$.

The covering $C'\ra C$ satisfies the descent for $F^f$-isocrystals. In other words,
\begin{lemma} \label{the descent}
$\cC$ can be identified as the category of 
\[\biggl\{(X', \{\varphi_\g\}_{\g\in\G})| X'\in \mathrm{ob\,}\cC' , \varphi_\g: \g^* X' \ra X', \varphi_{\g\g'}=\varphi_{\g'}\circ \varphi_{\g}\biggr\}.\]
\end{lemma}

\begin{proof}
The proof is an easy corollary of a more general theorem (\cite[Theorem 4.5]{Ogus}). Or we can directly compute as follows:  it suffices to show any such object $(X', \{\varphi_\g\}_{\g \in \G})$ can descend to $C$. $F$-isocrystal corresponds to  $(W, \nabla_W)$ on $C$ with a local Frobenius on $W_{B(k_0)}$. Obviously $\varphi_\g$ gives a descent datum for vector bundles on  $C'_{B(k_0)}$. Note a flat connection is equivalent to the descent data on the deRham space. Thus the \'etale covering $C' \ra C$ also satisfies the descent for connections. Thereby we have a bundle with connection $(V', \nabla_{V'})$ on $C_{B(k_0)}$. Similarly, locally the Frobenius can be descent to $V'$. Taking the intersection $V' \cap W$ gives a coherent sheaf over $C$ with connection. And taking a double dual gives a locally free sheaf.  
\end{proof}

Then $\cC \ra \cC'$ is just the forgetful functor $\{(X', \{\varphi_\g\}_{\g\in\G})\} \mapsto X'$. 

\begin{prop}
Notations as above, $1\ra G'_\univ \ra G_\univ \ra  \G$ is a exact sequence. 
\end{prop}

\begin{proof}
For any object $\cE' \in \cC'$, $\oplus_{\g\in\G} \g^* \cE'$ is invariant under $\G$ and thus can be descent to $C$. So any object in $F^f$-isco$(C')$ is a subquotient of some object in $F^f$-isoc$(C)$. By (\cite[2.21 (b)]{Deli} ), $G'_\univ$ is a subgroup of $G_\univ$. 

For the other side $G_\univ \ra \G$, consider the kernel of $\cC \ra \cC'$, i.e. the full subcategory $\cal{K}$ in $\cC$ consisted of objects with trivial image in $\cC'$. Then $\cal{K}$ is also Tannakian and by \ref{the descent} it consists of objects $(\cO^{\oplus n}_{C'}, \{\varphi_\g\}_{\g\in\G})$ where $\cO_{C'}$ denotes the trivial isocrystal over $C'$. For any such object, $\varphi$ induces a representation of $\G$ on $k^n$.  Therefore $\cal{K}$ is equivalent to $\Rep(\G)$. The inclusion $\cal{K} \ra \cC$ gives a morphism between groups $G_\univ \ra \G$. 

Now it suffices to show the exactness of $1 \ra G'_\univ \ra G_\univ \ra \G$ in the middle. Let $K$ be the kernel of $G_\univ \ra \G$. Since under $G_\univ \ra \G$, $G'_\univ$ has trivial image,  $G'_\univ \subset K$, i.e. we have 
\[\Rep(K) \ra \Rep(G'_\univ).\]
 For any $g\in G_\univ$ such that $g$ is in the kernel $K$, consider the object $(\oplus_{\g\in\G} \g^* \cE', \{\varphi_\g\})$ in $\cC$ and $\cE' \in \cC'$. The element $g$ fixes each direct summand $\g^* \cE$. In particular, let $\g=\mathrm{id}$ and $g$ acts on $\cE'$. Thus every $\cE' \in \Rep(G'_\univ)$ is  a natural representation of $K$. So $K=G'_\univ$.
\end{proof}

\begin{cor}  $\dim G'_\univ = \dim G_\univ$. \end{cor}

For the specific crystal $\cE$ in \ref{main thm}, viewed as $F^f$-crystal, $\cE$ has weight $f$. Let $E$ be the representation corresponding to $\cE$ and $G_E=\im(G_\univ \ra \Aut(E))$.

\begin{rmk}  \label{same dimension}
For $E$, we have 
\[\xymatrix{
G'_\univ \ar@{^{(}->}[dd] \ar[dr] & \\
&\Aut(E) \\
G_\univ \ar[ur] &
}\]
which induces $G'_E \hookrightarrow G_E$. Since $\dim G'_\univ = \dim G_\univ$, we have $\dim G'_E =\dim G_E$. 
\end{rmk}

\begin{lemma}\label{rep side}
Notation as above, we have $\dim_{\Q_{p^f}} (\wedge^4 E)^G=1$ and $\Q_{p^f} \otimes \End(\wedge^2 E)^G \cong \Q^{\times 4}_{p^f}$ as algebras.
\end{lemma}

\begin{proof}
Condition (2) is equivalent to $\dim_{\Q_{p^f}} \G((C/W)_\cri, \wedge^4 \cE(2f))^{F^f}\otimes \Q_{p^f}=1$. The space \[\G((C/W)_\cri, (\wedge^4 \cE)(2f))^{F^f}\]
 consists of invariant elements in $(\wedge^4 \cE ) (2f)$. The representation associated to $\wedge^4 \cE(2f)$ is $E\otimes \chi^{2f}$ where $\chi$ is the character of $G_\univ$ corresponding to Tate twist. Then $\dim_{\Q_{p^f}}(E\otimes \chi^{2f})^{G_\univ}=1$. Note Tate twist only affects the weight, $\dim_{\Q_{p^f}} (E)^G =\dim_{\Q_{p^f}} (E)^{G_\univ}=\dim_{\Q_{p^f}} (E\otimes \chi)^{G_\univ}=1$.
 
The isomorphism $\Q_{p^f} \otimes \End(\wedge^2 E)^G \cong \Q^{\times 4}_{p^f}$ follows directly.
\end{proof}
\subsection{Notations}We summarize the notations here. 
\begin{description}
\item[$\tilde C$]  \hfill \\
Any lifting of $C$ to $W(k_0)$.

\item[$\eta$]  \hfill \\
The generic point of $C$.
\item[$C' , \eta' $]  \hfill \\
 $C'$ an \'etale covering of $C$, $\eta'$ generic point of $C'$.
 Any symbol with a prime $'$ denotes its pullback from $C$ to $C'$, see \ref{notations on C'}.

\item[$\cal{E}$]  \hfill \\
 The crystalline higher direct image of $\pi$: $R^1\pi_{\cri, *}(\cal{O}_X)$, also the bundle with an integrable connection and Frobenius. 

\item[$G_\univ$]  \hfill \\
 The affine group scheme such that $F^f$-isoc$(C)$ is equivalent to $\Rep(G_\univ)$. 

\item[$E$] \hfill \\  The Tannakian duality of $\cE$ in $\Rep(G_\univ)$. 
\item[$G_E$]  \hfill \\ $\im(G_\univ \ra \Aut(E))$.
\item[$\frak{g}_E$] \hfill \\  The Lie algebra of $G_E$.
\item[$\cal{G}, \g$]  \hfill \\ A morphism $\g: \cE \ra \cG$ which is surjective after inverting $p$. 

\item[$G$]  \hfill \\ The BT group corresponding to $\cal{G}$. 
\item[$\r$]  \hfill \\ The morphism between BT groups $\r: G \ra X[p^\infty]$ corresponding to $\g: \cE \ra \cG$. 
\item[$K_n$]  \hfill \\ The $p^n$-torsion flat group scheme $\r_\eta(G_\eta)[p^n]^-$ on $C$.
\item[$\{\cT_n\}$] \hfill \\ A filtration $\{\cdots\cT_{n-1} \supset \cT_n \cdots\}$ with $\cT_n = \DD(X/K_n)$.
\item[$\{\cI_n\}$]  \hfill \\ A filtration $\{\cdots\cI_{n-1} \supset \cI_n \cdots\}$ with $\cI_n = \g^{-1}(p^n\cG)$.
\item[$H_n$] \hfill \\ $K_n/K_{n-1}$. 
\item[$i: W(k_0) \ra \tilde C$] \hfill \\ A section which lifts $\Spec k_0 \ra \eta$.
\item[$\cE_\eta$(or $i^*\cE$)]  \hfill \\the Dieudonne module over $W(k_0)$ which corresponds to $X[p^\infty]_\eta$.
\item[$\cG_\eta$(or $i^*(\cG)$)] \hfill \\  the Dieudonne module over $W(k_0)$ which corresponds to $G_\eta$ .
\item[$\cV_i$] \hfill \\ the rank 2 Dieudonne $F^f$-(iso)crystals appearing in the tensor decomposition of $\cE$ . 
\end{description}

\section{The structure of $G_E$} \label{a glance}

\subsection{Simplicity of $X/C$} 
\begin{prop}
There is no proper abelian subvariety $Y \hookrightarrow X$ over $C$.
\end{prop}

\begin{proof}
If there exists a proper abelian subvariety $Y \subset X$, then by Poincare irreducibility theorem, we have $Z \subset X$ and an isogeny $g: X \ra Y \times Z $.  The isogeny induces a morphism $g^*: H^4_\cri(Y\times Z, \cO) \ra H^4_\cri(X, \cO)$. By Kunneth formula for crystalline cohomology, 
\[H^4_\cri(Y\times Z, \cO)=H^4_\cri(Y)\oplus H^4_\cri(Z)\oplus H^2_\cri(Y)\otimes H^2_\cri(Z) \oplus \cdots .\]
 The self product of polarizations on $Y$ and $Z$ gives the nontrivial elements in the second cohomology groups which are all Frobenius eigen-elements. If $\dim Y \geq 2$, then $H^4_\cri(Y)^{F-p^2}$ contains the self product of the polarization. Thereby $H^4_\cri(Y\times Z, \cO)^{F-p^2}$ contains at least two linearly independent idempotents.

The Leray spectral sequence $H^4_\cri(Y\times Z, \cO) \ra H^0((C/W)_\cri, R^4\pi_* (\cO_{Y\times Z}))$ induces diagram 
\[\xymatrix{
H^4_\cri(Y\times Z, \cO)^{F^f}  \ar[r]^{g^*} \ar[d]^{\text{pr}} & H^4_\cri(X, \cO)^{F^f} \ar[d] \\
H^0((C/W)_\cri, R^4\pi_* (\cO_{Y\times Z}))^{F^f} \ar[r] & H^0((C/W)_\cri, R^4\pi_*(\cO_X))^{F^f}.\\
}\]
The idempotents in $H^4_\cri(Y\times Z, \cO)$ gives two linearly independent elements in $\G((C/W)_\cri, R^4\pi_*(\cO_X))^{F^f}$. Then $\dim_{\Q_p} \G((C/W)_\cri, R^4\pi_*(\cO_X))^{F^f} \geq 2$ for any $f$, contradicting to condition (3) in \ref{main thm}.
\end{proof}

\subsection{ $G_E$ reductive}
We will show that $G_E$ is reductive. The idea is to show $E$ is a faithful irreducible representation of $G_E$. Firstly we show the following general fact.

\begin{thm} \label{simple and irreducible}
The abelian scheme $X/C$ is simple if and only if  $\DD(X/C)$ is an irreducible $F$-isocrystal over $C$.
\end{thm}

%Let $\eta \in C$ be the generic point. If $X_\eta$ is not simple, then it is isogenous to a product of abelian varieties $Y \times Z$. Extend the $Y$ and $Z$ to the whole curve and apply the Kunneth formula. 
\subsection{The proof of \ref{simple and irreducible}}
 If $\cE$ is not irreducible, then there exists a $F$-isocrystal $\cal{G}$ such that $\cE \ra \cG$ is surjective. The slopes of $\cG$ are between 0 and 1. Hence $\cG$ has a model of $F$-crystal over $C/W(k)$ (see Appendix \ref{crystal model}), which we still denote as $\cG$. It is easy to show the Verschiebung $V$ also descends to $W(k)$ and hence $\cG$ is a Dieudonne crystal.

However, the morphism between Dieudonne crystals 
\[\g: \cE \ra \cG\] may not be surjective. We only know that $\im \g \supset p^k\cG$ for some integer $k$.

By \ref{equiv on curve}, $\g$ corresponds to a morphism between BT groups: 
\[\r: G \ra X[p^\infty].\]
Since $p^k \cG \subset \im \g$, the kernel of $\r$ is a subgroup scheme of $G[p^k]$. In particular, $\ker \r$ is finite. Let $\eta \in C$ be the generic point. Though $\im\r$ is merely an fppf abelian sheaf, its generic fiber $\im\r_\eta \subset X[p^\infty]_\eta$ is a BT group. Further, due to the following lemma, we can further assume $\r_\eta$ is injective. 

\begin{lemma} \label{generically injective}
The morphism $\r$ factors through a BT group $G'$ such that $G'\ra X[p^\infty]$ is generically injective. 
\end{lemma}

\begin{proof}

Let $K=(\ker \r_\eta)^-$ be the closure of $\ker \r_\eta$ in $G$. Then $K$ is a finite flat group scheme over $C$. Further, $\r(K)=0$ since $\r(K)_\eta=0$ and $K$ is flat. Therefore $K\subset \ker \r$. Let $G'=G/ K$. Then $G'$ is a BT group and $\r$ factors through $G'$. 
\end{proof}

Since $\r_\eta$ is injective, $\r_\eta(G_\eta[p^n])^-=\r_\eta(G_\eta)[p^n]^-$. We denote it as $K_n$. Then $K_n$ is a finite flat group scheme over $C$. Therefore $X[p^\infty]/K_n$ is a BT group and applying the crystalline Dieudonne functor yields that $\cal{T}_n:=\DD(X/K_n)$ is a finite locally free subsheaf of $\cE$.  And they form a filtration $\cT_0=\cal{E} \supset\cdots\cT_{n-1}\supset \cT_n \supset \cT_{n+1}\cdots$ with subquotient $\cT_{n-1}/\cT_n=\DD(K_n / K_{n-1})$. Another filtration on $\cE$ is that $\{\cal{I}_n=\gamma^{-1}(p^n\cG)\}$. 

\begin{lemma} $\cT_n \subset \cI_n$. \end{lemma}

\begin{proof}
On the side of BT groups, we have the following diagram: 
\[\xymatrix{
G[p^n] \ar@{^{(}->}[r] \ar@{^{(}->}[d] & K_n \ar@{^{(}->}[d]\\
G \ar@{->}[r]^\r \ar@{->>}[d] & X[p^\infty] \ar@{->>}[d] \\
G \ar[r] & X[p^\infty]/K_n .\\
}\] Dually, on the side of Dieudonne crystals, 
\[\xymatrix{
\DD(G[p^n])   & \DD(K_n) \ar@{->>}[l] \\
\cG  \ar@{->>}[u]& \cE \ar@{->>}[u] \ar@{->}[l]_\gamma\\
\cG   \ar@{^{(}->}[u]_{p^n}&\cT_n . \ar@{^{(}->}[u] \ar[l]\\
}\] The above diagram gives that the image of $\cT_n$ in $\DD(G)$, composing the upper and right arrows, is contained in $p^n\cG$. Therefore $\cT_n \subset \cI_n$.
\end{proof}

On one hand, we restrict the two diagrams above at the generic point $\eta$. Note $\gamma_\eta$ is surjective. 
\[\xymatrix{
\DD(G_\eta[p^n]) & \DD(K_{n,\eta}) \ar[l]_\cong\\
\cG_\eta \ar@{->>}[u] & \cE_\eta \ar@{->>}[u] \ar@{->>}[l]_{\g_\eta}\\
\cG_\eta \ar@{^{(}->}[u] & \cT_{n,\eta} \ar@{^{(}->}[u] \ar@{->>}[l]
}\]
Since $G_\eta[p^n]\cong K_{n,\eta}$, we have $\cE_\eta/\cT_{n,\eta}\cong \DD(G_\eta[p^n])\cong \cG_\eta/p^n\cG_\eta \cong \cE_\eta / \cI_{n, \eta}$. We already have $\cT_n \subset \cI_n$. Hence $\cT_{n,\eta} = \cI_{n,\eta}$. 

On the other hand, since $\DD(K_n)$ is $p^n$-torsion,  $p^n\cE \subset \cT_n$. Therefore $\cT_n \otimes \cO_{\tilde C}[\frac{1}{p}] \cong \cE \otimes \cO_{\tilde C}[\frac{1}{p}]$. In particular, $\cT_n \otimes \cO_{\tilde C}[\frac{1}{p}]\cong \cI_n \otimes \cO_{\tilde C}[\frac{1}{p}]$. 

Therefore, $\cT_n \subset \cI_n$ induces isomorphisms over generical fiber of $\tilde C$ and generic point of $C$. In particular, it is isomorphic on every height 1 points in $\tilde C$. 

\begin{lemma} \label{zero quotient}
Let $(D, \frak{m})$ be a regular local domain of dimension 2 and 
\[0\ra M \ra N \ra Q \ra 0\] be a short exact sequence of $D$-modules with $M$ finite free, $N$ torsion free and $\mathrm{supp}(Q)\subset \{\frak{m}\}$, then $Q=0$.
 \end{lemma}

\begin{proof}
Since $D$ is regular of dimension 2, we have the reflexive module $N^{\vee\vee}$ of $N$ is free(\cite[Chapter 2, Proposition 25]{Friedman}). The new short exact sequence $0\ra M \xrightarrow{T} N^{\vee\vee} \ra Q' \ra 0$ still satisfies that $\mathrm{supp}(Q')\subset \{\frak{m}\}$. Since $N$ is torsion free, still by (\cite[ Chapter 2, Corollary 21]{Friedman}), we have $N\subset N^{\vee\vee}$. Thereby $\mathrm{supp}(Q) \subset \mathrm{supp}(Q')$. Note both of $M$ and $N^{\vee\vee}$ are free of the same rank and hence $T$ can be represented as a square matrix with entries in $D$. So the support of $Q'$ is the zero set of $\det T$. But if $\det T$ is a nonunit, then the dimension of the zero set $Z(\det T)$ is of dimension 1. Hence $Q=0$. 
\end{proof}

\begin{cor} \label{same filtration}
$\cT_n=\cI_n$.
\end{cor}
\begin{proof}
Localize the injection $\cT_n \hookrightarrow \cI_n$ at each closed point of $\tilde C$ and apply lemma \ref{zero quotient}.
\end{proof}

\begin{prop} \label{H}
$H_n \cong H_{n+1}$ if $n$ large enough.
\end{prop}
\begin{proof}
 By \ref{same filtration}, we have $\g(\cT_n)=p^n \DD(G) \cap \g(\cE)$. By Artin-Rees lemma, there exists an integer $k$ such that $\g(\cT_n)=p^{n-k}\g(\cT_k)$ for any integer $n>k$.  Since via \ref{same filtration}, $\ker \g \subset \cT_n$ for each $n$, $\cT_{n-1}/\cT_n \cong \g(\cT_{n-1})/\g(\cT_n)\cong \g(\cT_k) \otimes \cO_{\tilde C}[\frac{1}{p}] \cong  \g(\cT_{n-2})/\g(\cT_{n-1}) \cong \cT_{n-2}/\cT_{n-1}.$ So $H_n \cong H_{n+1}$ for $n>k$.
\end{proof}

Fix the integer $k$ in the proof of \ref{H}. Let $\r_\eta(G_\eta)^-$ denote the union $\cup_n K_n=\cup_n \r_\eta(G_\eta[p^n])$ and $H$ denote $\r_\eta(G_\eta)^-/K_k$. Then we have 
\begin{prop}
$H$ is a BT group over $C$. 
\end{prop}

\begin{proof}
Obviously $\r_\eta(G_\eta)^-/K_k$ is $p$-torsion. And we have $\r_\eta(G_\eta)^-/K_k[p]=H_{k+1}$ is a finite locally free group scheme. It remains to show that $H$ is $p$-divisible. From Proposition \ref{H}, $H[p^2] \xrightarrow{p} H[p]$ is surjective. Now we proceed by induction. Suppose $H[p^n] \xrightarrow{p} H[p^{n-1}]$ is surjective. The following diagram is always commutative: 
\[\xymatrix{
H[p^n] \ar@{->>}[r]^p \ar@{^{(}->}[d] & H[p^{n-1}] \ar@{^{(}->}[d] \\
H[p^{n+1}] \ar[r]^p & H[p^n].
}\] 
Again by \ref{H}, the induced morphism on cokernels $H[p^{n+1}]/H[p^n] \ra H[p^n]/H[p^{n-1}]$ is isomorphic. Hence $H[p^{n+1}] \xrightarrow{p} H[p^n]$ is surjective. Therefore $H$ is $p$-divisible.
\end{proof}

Note $H$ is a subquotient BT group of $X[p^\infty]$. The following standard trick  follows from the proof of Theorem 2.6 in \cite{deJ2}.

Let $Z_k=X/K_k$. Then $Z_k$ is an abelian scheme over $C$ and $H$ is a sub BT group of $Z_k[p^\infty]$.  Put $Z'_n=Z_k/H[p^n]$. We have an exact sequence of truncated BT group schemes of level 1 over $C$ as follows
\[0\ra H[p] \ra Z_k[p] \ra Z'_n[p] \ra H[p] \ra 0.\] 
And hence an exact sequence 
\[0 \ra \o_H \ra \o_{Z'_n} \ra \o_{Z_k} \ra \o_H \ra 0.\]
We conclude that $\det(\o_{Z'_n})\cong \det(\o_{Z_k})$ independent of $n$. It is known (\cite{Zarhin}) that this implies there are only a finite number of isomorphism classes of abelian schemes among $Z'_n$. So we can find an abelian scheme $Z'_l$ such that there exists infinitely many $f_n \in \Hom(Z_k, Z'_l)$ such that $\ker f_n=H[p^n]$. Fix a $f_{n_0}$. And let $g_n \in \End(Z_k)$ be $ (f_{n_0})^{-1}\circ f_n $. Then $\ker g_n[p^\infty]$ is an extension of a fixed finite group scheme and $H[p^n]$. Let $g$ be the limit of $g_n$ in $\End(Z_k)$ and hence $\ker g[p^\infty]$ is an extension of a finite group scheme and $H$.  Therefore $\im g\subset Z_k$ is a proper subvariety. Since $Z_k$ and $X$ are isogenous, $g$ induces a morphism in $\End(X/C)$ which is not surjective. It contradicts to the assumption $X/C$ is simple. 

This is the end of the proof of \ref{simple and irreducible}.

Note by \ref{rep side}, we have
 \begin{equation} \label{*}
\dim_{B(k_0)}(\wedge^4 E)^{G_E}=1, \End(\wedge^2 E)^{G_E} \cong \Q^{\times 4}_{p^f}. \end{equation}

\begin{prop} \label{irreducibility1}
The representation $E$ is irreducible.
\end{prop}

\begin{proof}
If $E$ is not an irreducible $G_E$-representation, then $E$ has a proper sub-representation. Let $V$ be an irreducible sub-representation of $E$ with the smallest dimension. Then $V$ gives a proper sub object $\cV\subset \cE$ with minimal rank. Since $F: \cE^\s \ra \cE$ is an isomorphism between isocrystals, $F(\cV^\s)$ is also irreducible of the smallest rank and hence $F(\cV^\s) \cap \cV = 0$.  Consider $\sum_n F^n(\cV^{\s^n})$. It is a proper sub-isocrystal of $\cE$ and invariant under $F$. By \ref{simple and irreducible}, $\sum_n F^n(\cV^{\s^n}) =\cE$. As a quotient of the sum of irreducible elements, 
\[\cE\cong \oplus_{i\in I} \cV^{\s^i}\]
for some index set $I$. 

Therefore $E\cong \oplus_i V_i$. Since each $V_i$ has the same rank, the number of direct summands is either 1, 2, 4 or 8.  This number is not greater than 2 otherwise it violates $\End(\wedge^2 E)^{G_E} \cong \Q^{\times 4}_{p^f}$. If there are two direct summands, let 
\[E \cong V_1 \oplus V_2.\] Note $E$ admits a symplectic form $\l$ and $(\wedge^4 E)^{G_E}$ is generated by the self-product of $\l$. 

If $\l$ preserves the direct summands $V_1$ and $V_2$, then in the decomposition  
\[\wedge^4 E \cong \wedge^4 V_1 \oplus \wedge^4 V_2 \oplus \cdots,\]
 the self product $\l^2$ has nontrivial components in $\wedge^4 V_1$ and $\wedge^4 \cV_2$. Let $\l^2_1$ and $\l^2_2$ be the two components. Then both of them are invariant under $G_E$, contradicting to  $\dim_{B(k_0)}(\wedge^4 E)^{G_E}=1$. 
 
If the polarization does not preserve the direct summands, then $\l$ induces isomorphisms $V_1 \ra {V_2}^\vee$ and $V_2 \ra V^\vee_1$. Note $\wedge^2 E \cong \wedge^2 V_1 \oplus \wedge^2 V_2 \oplus V_1 \otimes V_2$ and there is a surjection 
\[S^2(\wedge^2 E) \ra \wedge ^4 E, w_1.w_2 \mapsto w_1 \wedge w_2.\]
 And then $\l^2$ lies in 
\[S^2(V_1 \otimes V_2) \cong S^2 V_1 \otimes S^2 V_2 \oplus \wedge^2 V_1 \otimes \wedge^2 V_2.\]
Let the $\l^2_1$ and $\l^2_2$ be the image $\l$ in the two components in $S^2(V_1 \otimes V_2)$. Then they are invariant under $G_E$, again contradicting to $\dim_{B(k_0)}(\wedge^4 E)^{G_E}=1$.
\end{proof}

\begin{cor}\label{irreducibility}
The group $G_E$ is reductive. 
\end{cor}

\begin{proof}
Note $G_E$ is reductive if $G_E$ admits a faithful and completely reducible representation.
\end{proof}
\subsection{$G_E$ nonsimple}

Since $X/C$ is principally polarized, the $F$-isocrystal $\cE$ admits a non-degenerate alternating form(\cite[Section 5.1]{BBM}). So $E$ also admits an alternating form which is preserved by $G_E$. Therefore the action of $G_E$ factors through $Sp(8, \Q_{p^f})$. Then the reductive Lie algebra $\frak{g}_E$ factors through $\frak{sp}(8)$. 

\begin{prop}
$G_E$ is not simple and the semisimple part $(\frak{g}_E)^{ss}_\C \cong \frak{sl}(2)^{\times 3}$.
\end{prop}

\begin{proof}

If $G_E$ is simple, then the Lie algebra $\frak{g}_E$ is a simple  subalgebra of $\frak{sp}(8, \Q_{p^f})$. Then condition \ref{*} can be stated in terms of Lie algebra:
\[\dim_{\Q_{p^f}}(\wedge^4 E)^{\frak{g}_E}=1, \End(\wedge^2 E)^{\frak{g}_E}  \cong \Q^{\times 4}_{p^f}.\]

Base change to $\C$. By Appendix \ref{simple groups}, there is no simple complex Lie algebras satisfying the conditions above. Therefore ${G_E}_\C$ is not simple. Then the semisimple part $\frak{g}^{ss}_\C=\frak{sl}(2)\times \frak{sl}(2) \times \frak{sl}(2)$ and $E_\C=E_1 \otimes E_2 \otimes E_3$ where $E_i$ is the standard representation of $\frak{sl}(2)$. Therefore 
\begin{multline}
\wedge^2 E_\C \cong (S^2 E_1 \otimes S^2 E_2 \otimes \wedge^2 E_3 )\oplus (\wedge^2 E_1 \otimes S^2 E_2 \otimes S^2 E_3 ) \oplus \\ (S^2 E_1 \otimes \wedge^2 E_2 \otimes S^2 E_3 )\oplus (\wedge^2 E_1 \otimes \wedge^2 E_2 \otimes \wedge^2 E_3)
\end{multline} as direct sum of irreducible representations. Therefore $\End_{G_E}(\wedge^2 E)_\C \cong \C^{\oplus 4}$ and $\wedge^2 E_1 \otimes \wedge^2 E_2 \otimes \wedge^2 E_3$ is the polarization. 

For $\frak{g}_E$,  $\End_\frak{g}(\wedge^2 E)\cong\Q^{\times 4}_{p^f}$. Therefore each of the 4 components of $\wedge^2 E_\C$ is defined over $\Q_{p^f}$. Note they are all nonfaithful nontrivial representations of $\frak{g}$. Thereby $\frak{g}$ and then $G_{E}$ is not simple. 
\end{proof}

So the Lie algebra $\frak{g}^{ss}_E=\frak{g}_1 \times \cdots \times \frak{g}_n$. Then $E=E_1\otimes \cdots\otimes E_n$, where $n\geq 2$, the $\frak{g}_i$ are simple Lie algebras and $E_i$ is a faithful representation of $\frak{g}_i$. Clearly at least one of the $E_i$ say $E_1$, has dimension 2, and this implies that $\frak{g}_1=\frak{sl}(2)$. Then $E_2 \otimes \cdots \otimes E_n$ is a 4-dimensional representation of $\frak{g}_2 \times \cdots \times \frak{g}_n$. Since $E_1$ is already symplectic, $E_2 \otimes \cdots \otimes E_n$ must be orthogonal. For simplicity, let $\frak{h}=\frak{g}_2 \times \cdots\times \frak{g}_n$ which acts orthogonally on $E_2 \otimes \cdots \otimes E_n$. Let us denote it as $W$.

\begin{cor}
\begin{equation} \label{the first decomposition}
\wedge^2 E\cong \Q_{p^f} \oplus W_1 \otimes W_2 \oplus W_1 \otimes W_3 \oplus W_2 \otimes W_3
\end{equation} over $\Q_{p^f}$ where each $W_i$ is a dimension 3 representation of $G_E$. Further $W_i$ admits a symmetric product. 
\end{cor}

\begin{proof}
Since $E=E_1 \otimes_{\Q_{p^f}} W$, 
$\wedge^2 E\cong S^2 W \oplus S^2 E_1 \otimes \wedge^2 W$. Let $W_1$ be $S^2 E_1$. Since the four idemponents in $\End_{G_E}(\wedge^2 E)_\C$ are defined over $\Q_{p^f}$, comparing with the decomposition over $\C$, $\wedge^2 W$ is the direct sum of two rank 3 representations, say $W_2$, $W_3$ and then 
\[\wedge^2 E = S^2 W \oplus W_1 \otimes W_2 \oplus W_2 \otimes W_3.\] 

As a subrepresentation of $\wedge^2 E$, $W_1 \otimes W_2$ admits a symmetric product. Therefore $S^2(W_1 \otimes W_2)$ has a one-dimensional trivial direct summand. Note 
\[\begin{aligned}
S^2(W_1\otimes W_2)&=S^2W_1\otimes S^2 W_2 \oplus \wedge^2 W_1 \otimes \wedge^2 W_2 \\
&= (\Q_{p^f} \oplus W'_1) \otimes S^2 W_2 \oplus \wedge^2 W_1 \otimes \wedge^2 W_2
\end{aligned}\] 
where $W'_1$ is irreducible. The 1-dimensional direct summand can only come from $\Q_{p^f}\otimes S^2 W_2$. Therefore $S^2 W_2$ has a one-dimensional trivial direct summand and hence $W_2$, as well as $W_3$, also admits a symmetric product. 
\end{proof}

However, we don't know whether $E_i$ are defined over $\Q_{p^f}$ yet. So we can not write $W_i$ as $S^2 V_i$ for some rank 2 representation $V_i$.  To remedy this situation, we take an \'etale covering of $C$ and increase the power $f$ if necessary.  

\section{The choice of the \'etale covering} \label{the choice}

In this section, we show how to choose the \'etale covering of $C$ for the existence of $V_i$. 

We consider a more general setting: if we have a rank 3 $F^f$-isocrystal $\cW$ with symmetric product, when can we write $\cW$ as $S^2 \cV$ for some rank 2 $F^f$-isocrystal $\cV$ which corresponds to a $SL(2)$- representation? If not in general, we compute the obstruction.

Firstly as crystals, $\cW$ corresponds to a bundle with an integrable connection over $\tilde C$. 

\begin{prop} \label{bundle}
The obstruction $o_1$ to the existence of a rank 2 bundle $\cV$ over $\tilde C$ such that $\cW =S^2 \cV[\frac{1}{p}]$ is in $H^2_{\mathrm{et}}(\tilde C, \m_2).$
\end{prop} 

\begin{proof}
Note the symmetric form $( , )$ gives an element in $\cO_{\PP(\cW)}(2)$. So it defines a conic bundle over $\tilde C$. If this conic bundle is isomorphic to $\PP(\cV)$ for some bundle $\cV$ on $\tilde C$, then $\cW\cong S^2 \cV$ and the symmetric product $(, )$ on $\cW$ is a scalar multiple of $S^2<,>$ where $<,>$ is an alternating form on $\cV$. So the obstruction is just the special Brauer class in $H^2_{\mathrm{et}}(\tilde C, \mu_2)$.
\end{proof} From the proof, if $\cW=S^2\cV$, we have $(,)=aS^2<,>$ for some $a\in \Q^*_{p^f}$. In particular, $(x^2,x^2)=0$ for any local section $x$ of $\cV$. 

\begin{rmk}
For the dimension 3 $PGL(2)$-representation $W$ and standard representation $V$ of $SL(2)$, $W=S^2V$ and there exists a 2-uple embedding $\PP(V)\cong \PP^1  \hookrightarrow \PP(W)$.
\end{rmk}

\begin{prop} \label{connection}
If $o_1=0$, then there exists a rank 2 bundle with connection $(\cV, \nabla_\cV)$ on $\tilde C$ such that $S^2\cV \cong \cW$ as modules with connection.
\end{prop}

\begin{proof}
If $o_1=0$, the let $\cV$ be the rank 2 bundle on $\tilde C$ such that $S^2 \cV \cong \cW$ over $\tilde C[\frac{1}{p}]$. For any affine open subset $U\subset \tilde C$ such that $\cV(U)$ and $\O^1_{\tilde C}(U)$ are free, suppose $T_{\tilde C} (U)$ is generated by $\t$. Let $N=\cV(U)$ and $M=\cW(U)$. For any section $x\in N$, $x^2 \in M$. Since $\nabla_\cW$ is compatible with the symmetric product and $(x^2,x^2)=0$, $((\nabla_\cW)_\t (x^2), x^2)=0$. Then $(\nabla_\cW)_\t (x^2)=x.v$ for some local section $v$ of $N[\frac{1}{p}]$. Define a map $\nabla_\cV$ locally as $(\nabla_\cV)_\t(x)=v$. It is easy to check $\nabla_\cV$ is a well-defined connection and it can be defined globally over $\tilde C[\frac{1}{p}]$. Further it is easy to show $S^2 \nabla_\cV = \nabla_\cW$.

From $S^2N[\frac{1}{p}] \cong M[\frac{1}{p}]$, there is an injective morphism $S^2N \ra M$. Now locally over $U$, let $N'=\sum_n (\nabla_\cV)^{(n)}_\t(N)$ and $(S^2N)'= \sum_n(\nabla_\cV)^{(n)}_\t (S^2N)$. Then $N \subset N' \subset N[\frac{1}{p}]$. Since $(S^2N)'\subset (S^2N)[\frac{1}{p}] \cap M$, $(S^2N)'$ is noetherian and hence there exists $k_0$ such that $(S^2N)' \subset \frac{1}{p^{k_0}} S^2N$.   

Next we prove $N'$ is also finitely generated. Choose generators $\{x, y\}$ of $N$. For simplicity, we use $\nabla$ to denote $(\nabla_\cV)_\t$. Then
\[\begin{pmatrix} \nabla x \\ \nabla y\end{pmatrix}= A \binom{x}{y} \] where $A$ is a 2$\times$2 matrix with entries in $\cO_U[\frac{1}{p}]$. Then roughly
\[\binom{\nabla^n x}{\nabla^n y}=(\nabla^n A + \cdots + A^n)\binom{x}{y}.\]

For any $z\in N'$(resp. $(S^2N)'$), let the order $\ord(z)$ of $z$ be the minimal integer $-k$ such that $z \in p^{-k}N $ (resp. $p^{-k} S^2 N$). Similarly, the order $\ord(A)$ of a matrix $A$ is the minimal integer among the orders of its entries. Let $A=p^{\ord(A)}\begin{pmatrix} a &b\\c&d \end{pmatrix}$.

We explore some identities of the order. Obviously $\ord(zz')=\ord(z)+\ord(z')$,  $\ord(AB)\leq \ord(A)+\ord(B)$. Since $\nabla$ commutes with $p$,  $\ord(\nabla^n A)=\ord(A)$ for any $n$ and $\ord(A\nabla A) = \ord(A^2)$.

If the sequence $\{\ord(A^n)\}$ is bounded below, then by the identities above, $\ord(\binom{\nabla^n x}{\nabla^n y})$ is also bounded below. Thus $N'$ is finitely generated. 

Now we assume $\{\ord(A^n)\}$ is not bounded below. By direct computation, it is easy to show the sequence $\{\ord(A^n)\}$ is strictly decreasing. Without loss of generality, we can assume $\ord(\nabla x) =\ord (A)$. Then $\ord((1,0) A^n)=\ord(A^n)$. Since
\[\nabla^n(x)=(1,0)( \nabla^n A + n A \nabla^n A + \cdots + A^n ) \binom{x}{y},\]  $\ord(\nabla^n x) = \ord(A^n)$. 
Consider $\nabla^n (x^2)=2\sum^n_{k=0} \nabla^k(x) \nabla^{n-k}(x) $.  The order of the product $\nabla^k(x)\nabla^{n-k}(x)$ is given by $(1,0) A^k \binom{x}{y}. (1,0)A^{n-k}\binom{x}{y}$. 

Let $A^k=p^{\ord(A^k)} \begin{pmatrix}a_k & b_k \\ c_k & d_k \end{pmatrix}$. Then 
\[\nabla^n(x^2)= [ (\sum a_k a_{n-k} )x^2 + 2(\sum a_k b_{n-k}) xy+ (\sum b_k b_{n-k}) y^2 ]p^{\ord(A^n)} + \text{lower order terms } .\]

Note $\nabla^n(x^2)\in (S^2 N)' \subset p^{k_0} S^2 N$. For $n$ large, the coefficient of $p^{\ord(A^n)}$ has to be zero. Since $x^2, xy, y^2$ are basis of $S^2N$, we have 
\[\sum a_k a_{n-k} =\sum a_kb_{n-k} = \sum b_k b_{n-k}=0.\]
Repeat the analysis for $\nabla^n(xy)$ and $\nabla^n(y^2)$.  We have for $n$ large, the following terms 
\[\sum a_k c_{n-k}, \sum b_k d_{n-k}, \sum c_kc_{n-k}, \sum d_kd_{n-k}, \sum c_kd_{n-k}, \sum(a_kd_{n-k}+b_kc_{n-k})\] are all zero. 
Note $A^n=p^{\ord(A^n)} \begin{pmatrix} \sum a_ka_{n-k}+ b_kc_{n-k} & \sum a_kb_{n-k}+b_kd_{n-k} \\ \sum c_ka_{n-k} + d_kc_{n-k} & \sum c_kb_{n-k}+d_kd_{n-k} \end{pmatrix}$ for any $k\leq n$. Thus 
\begin{multline}
nA^n=p^{\ord(A^n)} \begin{pmatrix} \sum_k a_ka_{n-k}+ b_kc_{n-k} &\sum_k a_kb_{n-k}+b_kd_{n-k} \\ \sum_k c_ka_{n-k} + d_kc_{n-k} & \sum_k c_kb_{n-k}+d_kd_{n-k} \end{pmatrix} \\ = p^{\ord(A^n)} \begin{pmatrix} \sum_k  b_kc_{n-k} &0 \\ 0 & \sum_k c_kb_{n-k} \end{pmatrix}.\end{multline}
In particular, for $n$ large, $A^n$ is always diagonal which implies $A$ has to be diagonal, i.e. $b_k =c_k=0$. Then $A^n$ is further the zero matrix, contradiction to the assumption $\{\ord(A^n)\}$ not bounded below.
%then there exists $K\in \mathbb{N}$ and $x \in N$ such that $\nabla^{(K)}_\cV(n) \notin \frac{1}{p^k} N$ while for any $k<K$, $\nabla^{(k)}_\cV(n) \in \frac{1}{p^k} N$ and $\frac{1}{p}m \notin M$. Then 
%\[ (S^2M)' \ni \nabla^{(K)}_\cW(n^2)=\sum_{k=0} \nabla^{(k)}_\cV (n) \nabla^{(K-k)}(n) \notin \frac{1}{p^k} S^2N ,\]contradiction. Therefore $N'$ is finitely generated. 

Hence $N'$ is finitely generated.

Since $\tilde C$ is a regular dimension 2 scheme, similar to the proof of \ref{zero quotient}, the double dual $N'^{\vee\vee}$ is locally free which also admits a connection.

Since all the arguments above are canonical, the existence holds globally.
\end{proof}

Therefore there is no obstruction to the connection. Now since the \'etale cohomology group $H^2_{\mathrm{et}}(\tilde C, \mu_2)$ is killed by any 2:1 \'etale covering of $C$. So over any 2:1 \'etale covering $C'$ of $C$, there exists a rank 2 crystal $\cV$ such that $S^2 \cV \cong \cW_{C'}$. Next we consider the Frobenius. 

\begin{prop} \label{F^f-isocrystal}
There exists an \'etale covering $C'\ra C$ such that over $C'$, there is a rank 2 $F^f$-isocrystal $\cV$ with $S^2\cV=\cW_{C'}$.
\end{prop}

\begin{proof}
By \ref{bundle}, \ref{connection}, we can choose any 2:1 covering killing the obstruction to the existence of a crystal. So we assume there exists crystal $\cV$ such that $S^2 \cV = \cW$ over $\tilde C$. 

Notations as in the proof of \ref{connection}, we choose an open affine subset $U =\Spec A\subset \tilde C$ and $N=\cV(U), M=\cW(U)$. So $S^2 N =M$. Let $\tilde \s$ be the lifting of $\s$ to $U$.  Note $\cW$ is an $F^f$-isocrystal. Without loss of generality, we can assume $f=1$ and then $F_M: M^{\tilde \s} \ra M$ and it is compatible with the symmetric forms $(,)$ and $(,)^{\tilde \s}$.

Shrinking $U$ if necessary, assume further $N$ is free and generated by $(x,y)$. Then correspondingly, $M$ is generated by $(x^2, xy, y^2)$. Since $(x^2, x^2)=0$, $(F_M(x^2), F_M(x^2))=0$. Thereby $F_M(x^2)$ is contained in the conic bundle. Because the obstruction $o_1$ is trivial, $F_M(x^2)$ comes from a square of some element in $N$. Thus $F_M(x^2)$ is also a square up to a scalar: $F_M(x^2)=\l x'^2$. Similarly, $F_M(y^2)=\mu y'^2$. Here $x'.y'$ are elements in $N$ and $\l, \mu \in A[\frac{1}{p}]^*$.  Since $(xy, x^2)=(xy, y^2)=0$, $F_M(xy)=\nu x'y'$. 

Note $x^2 + sxy+y^2$ is also a square in $M$ and thus 
\[F_M(x^2 + sxy+y^2)=\l x'^2 + 2\nu x'y' + \mu y'^2\]
is also a square up to scalar.  We can rewrite it as 
\begin{equation*}
\begin{aligned}
F_M(x^2)&=& \o x''^2\\
F_M(xy)&=& \o x''y''\\
F_M(y^2)&=& \o y''^2
\end{aligned}
\end{equation*} for some $\o\in A[\frac{1}{p}]^*$. Therefore we choose $F_N$ such that $F_N(x)=x''$ and $F_N(y)=y''$ and 
\[F_M=\o S^2F_N.\]

We also need the $F_N$ to be horizontal. Note 
\[\nabla_\cW(F_M(x^2))=\nabla_\cW(\o x^2)=d\o \otimes x'^2 + \o (2x') \nabla_\cV x' \] and 
\[\nabla_\cW(F_M(x^2))=F_M(\nabla_\cV(x^2))=F_M(2x\nabla_\cV x)=2\o F_N(x)F_N(\nabla_\cV x).\] Similarly, we can compute for $F_M(y^2)$. So $F_N$ commutes with $\nabla$ if $d\o=0$. 

Write $\o=p^v.u$ where $u \in A^*$. If $v$ is not even, then consider $F^2_M$ instead of $F_M$ in which case $\o$ is replaced by $\tilde \s(\o)\o=p^{2v}u'$. So we can assume $v$ is even.

For the unit $u$, there exists 2:1 \'etale covering $U' \xrightarrow{f} U$ (note $p>2$), such that $f^*(u)$ is a square in $A'^*$. Then $f^*(\o)=p^v f^*(u)$ is a square $\o'^2$.  Adjust $F_N$ such that $F_N(x)=\o' x'', F_N(y)=\o'y''$. Then 
\[F_M=S^2 F_N\]
and $F_N$ is compatible with the connection $\nabla_\cV$ and alternating form $<,>$.  

Now we consider the global case. For any affine covering $\tilde C=\cup_i U_i$, there exist 2:1 covering $U'_i \ra U_i$ such that we can find $F_{N,i}$ with $F_M=S^2F_{N,i}$.  Over $U'_i \times_C U'_j$, $S^2F_{N,i}=S^2F_{N,j}$ and thus $F_{N,i}=\tau F_{N,j}$ with $\t^2=1$, i.e.
\[\tau\in\m_2(A[\frac{1}{p}])=\m_2(A)=\m_2(A/p).\]

Therefore the obstruction of the existence of $F_\cV$ is in $H^1_{\mathrm{et}}(C,\m_2)$. This \'etale cohomology group can be killed by some 2:1 \'etale covering of $C$.
\end{proof}

Each $W_i$ in (\ref{the first decomposition}) corresponds to a rank 3 isocrystal $\cW_i$ over $C$ with a symmetric form. Hence we can choose a finite \'etale covering $C' \ra C$. The base change of $\cW_i$ to $C'$ are the second power symmetric product of some rank 2 isocrystals $S^2 \cV_i $. 

\begin{rmk} \label{notations on C'}
In the rest of the paper, we consider all the datum base change to $C'$ and use $\cE'$, $E'$, $G'_E$ to denote the pullback of $\cE, E, G_E$ to $C'$.
\end{rmk}

\section{The structure of $\cE'$ as an $F$-isocrystal} \label{the structure}

In this section, we will show that as  $F^f$-isocrystal over the \'etale covering $C'$ of $C$, $\cE$ has a tensor decomposition $\cV_1 \otimes \cV_2 \otimes \cV_3$. 

In the Tannakian formalism, the representation $V_i$ corresponding to $ \cV_i$ is dimension 2 with an alternating form and $S^2V_i=W_i$. Therefore $G'_\univ$ acting on $V_i$ factors through $SL(2)$ and then the Tannakian group of $\cV_1\otimes \cV_2 \otimes \cV_3$ is $\im(G'_\univ \ra SL(2)^{\times 3})$.  Now (\ref{the first decomposition}) transforms to 
\begin{equation} \label{the second decomposition}
\wedge^2 E' \cong \wedge^2 V_1 \otimes \wedge^2 V_2 \otimes \wedge ^2 V_3 \oplus S^2V_1 \otimes S^2 V_2 \otimes \wedge^2 V_3 \oplus S^2V_1 \otimes \wedge^2V_2 \otimes S^2 V_3 \oplus \wedge^2 V_1 \otimes S^2 V_2 \otimes S^2V_3.
\end{equation}
Thus we have the commutative diagram:

\[\xymatrix{
&G'_E \ar[dr] & \\
G'_\univ \ar[ur] \ar[dr] & & PGL(2)^{\times 3}\\
& SL(2)^{\times 3} \ar[ur]\\
}\]

\begin{lemma}
The map $G'_E \ra PGL(2)^{\times 3} $ is surjective. 
\end{lemma}

\begin{proof}
From (\ref{the first decomposition}), we know that $\ker(G'_E \ra PGL(2)^{\times 3})=\ker(G'_E \ra \Aut(\wedge^2 E'))$. Since $E'$ is a faithful representation of $G'_E$, the kernel is just $\pm I$. 

By \ref{same dimension}, since $\dim \frak{g}^{ss}_\C=9$, $\dim G'^{der}_E =\dim G^{der}_E = 9$. Hence $\im(G'_E \ra PGL(2)^{\times 3})$ has dimension 9. Since $PGL(2)$ is simple, it must be surjective. 
%Since $W_i$ are all $G'_E$ representation, $W_1\otimes W_2\otimes W_3$ is in $\Rep(G'_E)$ which induces a natural morphism between groups:
%%\[G'_E \ra PGL(2)^{\times 3}.\] 
%Note the composition of this morphism with projection to any two factors yields a %surjection $G'_E \ra PGL(2)^{\times 2}$. Thus for any $(g_1, g_2,g_3)\in PGL(2)^{\times 3}$,  
\end{proof}

Since $SL(2)$ is simply connected as an algebraic group and both of $SL(2)^{\times 3}$ and $G'^{der}_E$ are finite covering of $PGL(2)^{\times 3}$, there is a natural surjection $SL(2)^{\times 3} \ra G'^{der}_{E}$ with kernel isomorphic to $\m_2 \times \m_2$. Then 
\[\im(SL(2)^{\times 3} \ra \Aut(V_1 \otimes V_2 \otimes V_3))\cong G'_E.\] 

Since the action of $G'^{der}$ on $E'$ is absolutely irreducible, the center $Z(G'_E)$ acts as scalars on $E'$. So the action of $G'_E$ on $E'$ is the same as $SL(2)^{\times 3}$ up to some power of determinant. Let $\c$ be the character.

 Now we have two morphisms from $G'_\univ$ to $G'_E$, one induced by $\cV_1\otimes \cV_2 \otimes \cV_3$ and the other induced by $\cE'$. Denote them by $f_1, f_2$ respectively: 
\[\xymatrix{
G'_\univ \ar@/^/[r]^{f_1} \ar@/_/[r]_{f_2} & G'_E.
}\]
For any $g\in G'_\univ$, $f_1(g)=\c(g) f_2(g)$, i.e. the image of 
\[G'_\univ \xrightarrow{f_1\times f_2} G'_E \times G'_E\] is contained in $\D\cup \D'$ where $\D'=\{(g,\c(g)g)| g\in G'_E\}.$ The image is isomorphic to a subgroup of $G'_E \times \GG_m$. 

\begin{prop}
There exists an $F^f$- isocrystal $\cL$ such that 
\[\cV_1\otimes \cV_2\otimes \cV_3 \cong \cE' \otimes \cL\] as $F^f$-isocrystals. 
\end{prop}

\begin{proof}
Consider the sub Tannakian category generated by $\{\cV_1\otimes \cV_2 \otimes \cV_3, \cE'\}$. The group corresponds to this sub category is given by 
\[\im(G'_\univ \ra \Aut(V_1\otimes V_2 \otimes V_3)\times G'_E) \hookrightarrow G'_E \times \GG_m.\] 
Hence we have 
\[\xymatrix{
 &&G'_E\\
G'_\univ \ar[r] & G'_E \times \GG_m \ar[ur]^{\text{pr}} \ar[dr]_{m} \\
&& G'_E
}\] where $\text{pr}$ is the first factor projection while $m$ is the multiplication.  Through $m$, $\cV_1\otimes \cV_2 \otimes \cV_3\cong \cE'\otimes \cL$ where $\cL$ is a rank 1 $F^f$-isocrystal. 
\end{proof}

Replace $\cV_3$ by $\cV_3 \otimes \cL$ and we still denote it as $\cV_3$. Summarize the results and we have the following theorem.

\begin{thm} \label{tensor decomposition}

There exists a finite \'etale covering $C'$ of $C$ such that after base change to $\cri(C'/W(k))$, we have an isomorphism of $F^f$-isocrystals: 
\[ \cV_1 \otimes \cV_2 \otimes \cV_3 \cong \cE'  \] where $\text{rank } \cV_i=2$.
\end{thm}

\section{The end of the proof of \ref{main thm}} \label{the end}

Because of maximal Higgs field, $X' \ra C'$ is not isotrivial. Since $V_i$ are $SL(2)$-representations in $\Rep(G'_\univ)$, in particular, $\cV_i$ are all irreducible $F^f$-isocrystals,  we can apply the following general theorem to our $X' \ra C'$. 

\begin{thm} (\cite{Xia3}) \label{the 2nd step}
Let $X \ra C$ be a principally polarized abelian varieties of dimension $2^m$ over a smooth proper curve $C$ over an algebraically closed field $k$ of characteristic $p$.  Assume $p>2$ and 
\begin{enumerate}
\item $X_c$ is ordinary for some closed point $c\in C$,
\item $\cE \cong \cV_1 \otimes \cV_2 \otimes \cdots \cV_{m+1}$ as isocrystals where all $\cV_i$ are irreducible of rank 2,
\end{enumerate} then $X\ra C$ is a weak Mumford curve.

If we further assume the Higgs field associated to $X \ra C$ is maximal,
then there exists another family of fourfolds  $Y' \ra C'$ such that
\begin{enumerate}[(a)]
\item $Y' \ra X'$ is an isogeny over $C'$, 
\item $Y' \ra C'$ is a good reduction of a Mumford curve . 
\end{enumerate}

\end{thm}

In our case, though all $\cV_i$ are only irreducible as $F^f$-isocrystal,  according to Remark 5.5 in \cite{Xia3}, $\cE'$ is also isogenous to a tensor product $\cV \otimes \cT$ where $\cV$ is a rank 2 Dieudonne crystal with maximal Higgs field and $\cT$ is a rank 4 unit root crystal.  By \ref{equiv on curve}, there exists an isogeny $Y \ra X'$.  The rest arguments follow identically as those in Section 6 of \cite{Xia3}.

\section{A variation of \ref{main thm}} \label{a variation}
We have a variation of \ref{main thm}. Assumption and notation as $\ref{main thm}$ except that $k$ can be any algebraically closed field of characteristic $p$. Let $B(k)$ be the fractional field of the Witt ring $W(k)$.

\begin{thm} \label{the variation}
Assume $p>2$ and 
\begin{enumerate}
\item $X_c$ is ordinary for some closed point $c\in C$,
\item $\G((C/W(k))_\cri, \wedge^4 \cE) \otimes B(k)$ has dimension 1,
\item  $\G((C/W(k))_\cri, \scr{E}nd(\wedge^2 \cE)) \otimes B(k)\cong B(k)^{\times 4}$, 
\end{enumerate}
then $X\ra C$ is a weak Mumford curve. If the Higgs field of $\cE$ is maximal,
then there exists a family of abelian fourfolds $Y\ra C'$ such that 
\begin{enumerate} [(a)]
\item $C' \ra C$ is a finite \'etale covering,
\item $Y \ra X'$ is an isogeny over $C'$, between $Y$ and the pullback family of $X$,
\item $Y \ra C'$ is a good reduction of a Mumford curve.
\end{enumerate}\end{thm}

Note in Section \ref{example}, $\cE\cong \cV_1 \otimes \cV_2 \otimes \cV_3$ as isocrystals over $\cri(C/\Z_p)$.  Each $\cV_i$ corresponds to a $SL(2)$-representation. Thereby the computations in Section \ref{example} shows that specific good reductions of Mumford curves also serve as examples of \ref{the variation}. 
The proof adopts exactly the same method as the proof of \ref{main thm}, except that instead of $F^f$-isocrystals we consider the neutral Tannakian category of isocrystals over $C$. Then $\cE$ corresponds to a $B(k)$-representation $E$. Similarly, Conditions (2) and (3) in \ref{the variation} imply \begin{equation} 
\dim_{B(k)}(\wedge^4 E)^{G_E}=1, \End(\wedge^2 E)^{G_E} \cong \Q^{\times 4}_{p^f}. \end{equation}

 Note in the proof of \ref{irreducibility1}, we only use $(\ref{*})$. So the same result holds for $E$ and the Tannakian group corresponding to $\cE$ is still reductive. 
 
 The rest arguments through Section \ref{a glance}, \ref{the choice} and \ref{the structure} work with only change of the base field from $\Q_{p^f}$ to $B(k)$.  So we have the tensor decomposition as isocrystals
\[\cE' \cong \cV_1 \otimes \cV_2 \otimes \cV_3.\]
Then directly apply \ref{the 2nd step} and we have \ref{the variation}.
%viewing $F: \cE'^\s \ra \cE'$ as a morphism between $F^f-$isocrystals over $C'$, we can still mimic the argument preceding Proposition 4.3 in \cite{Xia3} to obtain the tensor decomposition of $F$, 
%% Add an appendix referring to the proof%%
%and further $\cE'$ is isogenous to a tensor product $\cV \otimes \cT$ where $\cV$ is a rank 2 Dieudonne crystal with maximal Higgs field and $\cT$ is a rank 4 unit root crystal.  By \ref{equiv on curve}, there exists an isogeny $Y \ra X'$.  The following arguments are identical in those in Section 6 of \cite{Xia3}.

\begin{appendix}
\section{} \label{simple groups}
%\section{Complex simple algebraic group with symplectic 8 dimensional representation}

We classify all complex simple Lie algebra $\frak{g}$ with an irreducible symplectic 8 dimensional representation. In other words, we look for an embedding of $\frak{g} \ra \frak{sp}(8)$ such that the standard representation of $\frak{sp}(8)$ is $\frak{g}$-irreducible. 

Since $\frak{g}$ is simple, $\dim \frak{g}\geq 2$. Note $\mathrm{rank}(\frak{g})\leq \mathrm{rank}(\frak{sp}(8))=4$. 

If $\mathrm{rank} (\frak{g})=4$, the adjusting by a conjugation, we can assume the embedding $\frak{g} \ra \frak{sp}(8)$ maps the Cartan subalgebra of $\frak{g}$ into Cartan subalgebra of $\frak{sp}(8)$, the direct sum of positive roots $\frak{g}^+$ to $\frak{sp}(8)^+$, $\frak{g}^-$ to $\frak{sp}(8)^-$. Hence each root space of $\frak{g}$ maps to a root space of $\frak{sp}(8)$ which induces a map between Dynkin diagrams.  Comparing the Dynkin diagrams of $A_4, B_4, D_4$ with $\frak{sp}(8)$ yields that none of them can be embedded into $\frak{sp}(8)$. 

Therefore the only possible Lie algebras are $A_1, A_2, A_3, B_2, B_3, C_3, C_4$(note $\frak{sl}(4)\cong \frak{so}(6), \frak{sp}(4)\cong \frak{so}(5)$).

In each of the following cases, let $V$ always denote the standard representation of corresponding Lie algebras.
\begin{enumerate}
 
\item $A_1=\frak{sl}(2)$.

The unique 8 dimension irreducible representation of $\frak{sl}(2)$ is the symmetric power $S^7V$. Since $\frak{sl}(2)\cong \frak{sp}(1)$, $S^7 V$ is symplectic.

Now consider $\wedge^4(S^7 V)^{\frak{sl}(2)}$. By (\cite[11.35]{Fulton}), $\wedge^4(S^7 V) \cong S^4(S^4 V)$. Counting the dimension of each weight spaces yields the decomposition 
\[S^4(S^4 V) \cong V_{16}\oplus V_{12}\oplus V_{10}\oplus V_8 \oplus V_6 \oplus V^{\oplus 2}_4 \oplus V^{\oplus 3}_0.\]Therefore $\dim_\C \wedge^4(S^7 V)^{\frak{sl}(2)}=3$ which is too big.

\item $A_2=\frak{sl}(3)$

By (\cite[15.17]{Fulton}), for any irreducible representation $\G_{a,b}$ with highest weight $aL_1-bL_3$, the dimension $\dim_\C \G_{a,b}=(a+b+2)(a+1)(b+1)/2$. Then $\G_{a,b}$ is 8-dimensional if and only if $a=b=1$. Note $\G_{1,1}$ is nothing but the adjoint representation of $\frak{sl}(3)$ which is the traceless subrepresentation of $\End(V)$. So it has a nondegenerate symmetric, not alternating form. Therefore $\frak{sl}(3)$ does not have a symplectic irreducible representation. 

\item $A_3=\frak{sl}(4)$

Still by (\cite[15.17]{Fulton}), $\dim \G_{a_1,a_2,a_3}=(a_1+1)(a_2+1)(a_3+1)(a_1+a_2+2)(a_2+a_3+2)(a_1+a_2+a_3+3)/12$ where each $a_i$ is a nonnegative integer. No such $a_i$ makes $\dim \G_{a_1,a_2,a_3}=8$. Hence $\frak{sl}(4)$ has no 8-dimensional irreducible representation. 

\item $B_2=\frak{so}(5)$

By (\cite[24.30]{Fulton}), $\dim \G_{a_1,a_2}=(a_1+1)(a_1+a_2+2)(2a_1+a_2+3)(a_2+1)/6$. No $a_i$ makes it dimension 8. So $B_2$ has no irreducible representation of dimension 8.

\item $B_3=\frak{so}(7)$ 

Again by (\cite[24.30]{Fulton}), 
$\dim \G_{a_1,a_2,a_3}=(a_1+1)(a_3+1)^2(a_1+a_2+2)(a_1+2a_2+a_3+4)(a_1+a_2+a_3+3)(a_2+a_3+2)(2a_1+2a_2+a_3+5)(2a_2+a_3+3)/720.$ Still no $a_i$ make it 8. Therefore $B_3$ has no 8-dimensional irreducible representation.

\item $C_3=\frak{sp}(6)$

By (\cite[24.20]{Fulton}), $\dim \G_{a_1, a_2,a_3}=(a_3+1)(a_2+a_3+2)(a_1+a_2+a_3+3)(a_1+1)(a_2+1)(a_1+a_2+2)(a_1+2a_2+2a_3+5)(a_1+a_2+2a_3+4)(a_2+2a_3+3)/720$. No $a_i$ make it 8. Hence $\frak{sp}(6)$ has no irreducible 8-dimensional representation. 

\item $C_4=\frak{sp}(8)$

Then the second exterior product of standard representation $\wedge^2 V$ decomposes to $\wedge^2 V \cong \C \oplus W$ with $W$ irreducible $\frak{sp}(8)$-representation. So $\End_{\frak{sp}(8)}(\wedge^2 V)=2<3$. 

\end{enumerate}

\section{crystal model} \label{crystal model}

Assume we have a $F$-isocrystal $\cV$ over $C$. Then it is in prior a crystal over $C$. Locally, it corresponds to a module with connection $(M. \nabla)$ over $\tilde C$ such that $M \otimes B(k)$ admits a $\s$-linear morphism $F$. The point is to descend the Frobenius $F$.

Though $F$ may not descend to $M$, we can consider $M'=\sum_n F^{(n)}(M)$. Then $M \subset M' \subset M\otimes B(k_0)$. One can mimic the the proof of ( \cite[Theorem 2.6.1]{Katz}) to show $M'$ is finitely generated provided that all slopes are nonnegative. 

Since $\tilde C$ is regular of dimension 2, taking the double dual $M'^{\vee\vee}$ gives a locally free sheaf over $\tilde C$. 

If $\cV$ is an isocrystal with slopes between 0 and 1, then we can further choose a morphism $V: \cV\ra \cV^\s$ such that $V\circ F= F\circ V =p.$ Hence such isocrystal has a model of Dieudonne crystal.

\end{appendix}

\bibliographystyle{amsplain}
\bibliography{mybib}{}

\providecommand{\bysame}{\leavevmode\hbox to3em{\hrulefill}\thinspace}
\providecommand{\MR}{\relax\ifhmode\unskip\space\fi MR }
\providecommand{\MRhref}[2]{%
  \href{http://www.ams.org/mathscinet-getitem?mr=#1}{#2}
}
\providecommand{\href}[2]{#2}
\begin{thebibliography}{10}

\bibitem{BBM}
Pierre Berthelot, Lawrence Breen, and William Messing, \emph{Th\'eorie de
  {D}ieudonn\'e cristalline. {II}}, Lecture Notes in Mathematics, vol. 930,
  Springer-Verlag, Berlin, 1982. \MR{667344 (85k:14023)}

\bibitem{deJ}
A.~J. de~Jong, \emph{Crystalline {D}ieudonn\'e module theory via formal and
  rigid geometry}, Inst. Hautes \'Etudes Sci. Publ. Math. (1995), no.~82, 5--96
  (1996). \MR{1383213 (97f:14047)}

\bibitem{deJ2}
\bysame, \emph{Homomorphisms of {B}arsotti-{T}ate groups and crystals in
  positive characteristic}, Invent. Math. \textbf{134} (1998), no.~2, 301--333.
  \MR{1650324 (2000f:14070a)}

\bibitem{Deli}
P.~Deligne, \emph{Cat\'egories tannakiennes}, The {G}rothendieck {F}estschrift,
  {V}ol.\ {II}, Progr. Math., vol.~87, Birkh\"auser Boston, Boston, MA, 1990,
  pp.~111--195. \MR{1106898 (92d:14002)}

\bibitem{Hodgecycles}
Pierre Deligne, James~S. Milne, Arthur Ogus, and Kuang-yen Shih, \emph{Hodge
  cycles, motives, and {S}himura varieties}, Lecture Notes in Mathematics, vol.
  900, Springer-Verlag, Berlin, 1982. \MR{654325 (84m:14046)}

\bibitem{Friedman}
Robert Friedman, \emph{Algebraic surfaces and holomorphic vector bundles},
  Universitext, Springer-Verlag, New York, 1998. \MR{1600388 (99c:14056)}

\bibitem{Fulton}
William Fulton and Joe Harris, \emph{Representation theory}, Graduate Texts in
  Mathematics, vol. 129, Springer-Verlag, New York, 1991, A first course,
  Readings in Mathematics. \MR{1153249 (93a:20069)}

\bibitem{Katz}
Nicholas~M. Katz, \emph{Slope filtration of {$F$}-crystals}, Journ\'ees de
  {G}\'eom\'etrie {A}lg\'ebrique de {R}ennes ({R}ennes, 1978), {V}ol. {I},
  Ast\'erisque, vol.~63, Soc. Math. France, Paris, 1979, pp.~113--163.
  \MR{563463 (81i:14014)}

\bibitem{Moller}
Martin M{\"o}ller, Eckart Viehweg, and Kang Zuo, \emph{Special families of
  curves, of abelian varieties, and of certain minimal manifolds over curves},
  Global aspects of complex geometry, Springer, Berlin, 2006, pp.~417--450.
  \MR{2264111 (2007k:14054)}

\bibitem{Moon}
B.~J.~J. Moonen and Yu.~G. Zarhin, \emph{Hodge classes and {T}ate classes on
  simple abelian fourfolds}, Duke Math. J. \textbf{77} (1995), no.~3, 553--581.
  \MR{1324634 (96b:14010)}

\bibitem{Mum}
D.~Mumford, \emph{A note of {S}himura's paper ``{D}iscontinuous groups and
  abelian varieties''}, Math. Ann. \textbf{181} (1969), 345--351. \MR{0248146
  (40 \#1400)}

\bibitem{famofav}
David Mumford, \emph{Families of abelian varieties}, Algebraic {G}roups and
  {D}iscontinuous {S}ubgroups ({P}roc. {S}ympos. {P}ure {M}ath., {B}oulder,
  {C}olo., 1965), Amer. Math. Soc., Providence, R.I., 1966, pp.~347--351.
  \MR{0206003 (34 \#5828)}

\bibitem{Ogus}
Arthur Ogus, \emph{{$F$}-isocrystals and de {R}ham cohomology. {II}.
  {C}onvergent isocrystals}, Duke Math. J. \textbf{51} (1984), no.~4, 765--850.
  \MR{771383 (86j:14012)}

\bibitem{JX}
J~Xia, \emph{A characterization of the good reduction of mumford curve}.

\bibitem{Xia3}
Jie Xia, \emph{On the shimura curves in positive characteristic}.

\bibitem{Zarhin}
Ju.~G. Zarhin, \emph{Isogenies of abelian varieties over fields of finite
  characteristic}, Mat. Sb. (N.S.) \textbf{95(137)} (1974), 461--470, 472.
  \MR{0354685 (50 \#7162b)}

\end{thebibliography}

\end{document}